\documentclass[12pt]{amsart}
\usepackage{amssymb,amsmath,amsthm}%,stmaryrd}
\usepackage{verbatim}
\usepackage{graphicx}
\usepackage{epsfig, enumerate}
\usepackage{color}
\usepackage{enumerate}
\usepackage{soul}
\usepackage{dsfont}
\usepackage[all]{xy}

\DeclareMathAlphabet{\mathpzc}{OT1}{pzc}{m}{it}
\newcommand{\N}{{\ensuremath{\mathbb{N}}}}

\def\A{\mathcal A}
\def\B{\mathcal B}
\def\F{\mathcal F}

\def\I{\mathcal I}
\def\J{\mathcal J}
\def\K{\mathcal K}
\def\L{\mathcal L}

\def\N{\mathcal N}

\def\overF{\overline{\mathcal F}}

\newcommand{\norm}[1]{\left\lVert #1 \right\rVert}
\def\Ae{\mbox{\AE}}
\def\d{\displaystyle}

\def\Lip{{\rm Lip}}
\def\FIN{{\rm FIN}}
\def\MFIN{{\rm MFIN}}

\def\LIN{{\rm LIN}}
\def\LIP{{\rm LIP}}
\def\CLIN{{\mathfrak {A}}}
\def\CLIP{{\mathfrak{I}}}
\newtheorem{theorem}{Theorem}[section]
\newtheorem{lemma}[theorem]{Lemma}
\newtheorem{definition}[theorem]{Definition}
\newtheorem{corollary}[theorem]{Corollary}
\newtheorem{proposition}[theorem]{Proposition}

\newtheorem{problem}[theorem]{Problem}

\begin{document}
\title[Galois connections and minimal Lipschitz operator ideals]{Galois connection between Lipschitz and linear operator ideals and minimal Lipschitz operator ideals}

\author{Pablo Turco}
\address{IMAS - UBA - CONICET - Pab I,
 Facultad de Cs. Exactas y Naturales, Universidad de Buenos
Aires, (1428) Buenos Aires, Argentina}
\email{paturco@dm.uba.ar}
\author{Rom\'an Villafa\~ne}
\address{Departamento de Matem\'atica - Pab I,
 Facultad de Cs. Exactas y Naturales, Universidad de Buenos
Aires, (1428) Buenos Aires, Argentina}
\email{ rvillafa@dm.uba.ar}


\thanks{CONICET PIP 2014--0483, ANPCyT PICT 2015--2299 }

\begin{abstract}
We establish a relation between Lipschitz operator ideals and linear operator ideals, which fits in the framework of Galois connection between lattices. We use this relationship to give a criterion which allow us to recognize when a Banach Lipschitz operator ideal is of composition type or not. Also, we introduce the concept of minimal Banach Lipschitz operator ideal, which have analogous properties to minimal Banach operator ideals. Also we characterize minimal Banach Lipschitz operator ideals which are of composition type and present examples which are not of this class.
\end{abstract}

\maketitle
\section{Introduction}

In 2003, Farmer and Johnson \cite{FaJo09} extend the notion of $p$-summing operators to the Lipschitz setting, introducing the concept of Lipschitz $p$-summing operators. The Lipschitz $p$-summing operators have similar properties respect the linear $p$-summing operators such as a (nonlinear) Pietsch factorization theorem \cite[Theorem~1]{FaJo09} and a extrapolation Theorem \cite[Theorem~2.2]{CheZhe11}. Also, the space of Lipschitz $p$-summing operators from a metric space to a dual Banach space is a dual Banach space \cite[Theorem~4.3]{Cha11}. The Lipschitz $p$-summing operators can be seen as the first of a large list of different classes of Lipschitz operators that had been studied in the last years. Most of these classes of Lipschitz operators are obtained as a generalization of linear operators, for instance the finite rank, approximable and compact operators \cite{JiSeVi14}, $p$-nuclear and $p$-integral operators \cite{CheZhe12}, among many others. The necessity to study this different classes in a general framework, unifying results and the language, leads to the new concept of Banach Lipschitz operator ideal defined in \cite[Definition~2.1]{ARSY16} and independently in \cite[Definition~2.3]{CCJV16} (under the name of {\it generic Lipschitz operator Banach ideal}).

From a Banach linear operator ideal there is a way to obtain a Banach Lipschitz operator ideal using the linearization of a Lipschitz map via the Arens-Ells space (see definition below). This Banach Lipschitz operator ideals are called {\it of composition type} and, in general, the properties of the Banach linear operator ideal can be transferred to the respective Banach Lipschitz operator ideal. This leads us to the first objective of our project: Determine whenever a Banach Lipschitz operator ideal is obtained from a linear operator ideal via a {\it  composition method}.

On the other hand, since linear operators are in particular Lipschitz operators, from a Banach Lipschitz operator ideal we may obtain a Banach linear operator ideal. In Section~\ref{Sec: Galois}, we show that this two procedures relates Lipschitz and linear operator ideals and fits in the framework of {\it Galois connection between lattices} (Theorem~\ref{Teo: Galois}). This relation between both types of ideals in the context of lattices let us to establish a criterion to identify if a Banach Lipschitz operator ideal is of composition type or not (Proposition~\ref{Prop: Criterion}). We apply the Criterion to give some examples of Banach Lipschitz operator ideals which are not of composition type (Porposition~\ref{Prop: No compo}).

Also, this point of view allow us to transfer some well known properties and particularities of the theory of Banach linear operators ideals into the Lipschitz setting. In Section \ref{Sec: minimal} we define the notion of \textit{minimal} Banach Lipschitz operator ideal, which as in the linear case, is related with the concept \textit{maximal hull} of a Banach Lipshcitz operator ideal, defined by Cabrera--Padilla, Ch\'avez--Dom\'{\i}nguez, Jim\'enez--Vargas and Villegas--Vallecillos \cite{CCJV16}. We characterize the minimal kernel and the maximal hull of Banach Lipschitz operator ideals which are of composition type (see Proposition~\ref{prop: Comp Maximal} and Theorem~\ref{Teo:min}). Finally, we give some properties of minimal Banach Lipschitz operator ideal and give an example in which the minimal kernel of a Banach Lipschitz operator ideal is not of composition type.

We refer the reader to the book of Weaver \cite{Wea99} for the basics of Lipschitz operators, to the book of Pietsch \cite{Pie80} for a general background of Banach linear operator ideals and to the book of Davey and Priestley \cite{DP02} for the theory of lattices.

\section{Notation and basics}

Along the manuscript, $X$ will be a metric space and $d$ its distance.  With $E$ and $F$ will denote Banach spaces with the norm $\norm{\cdot}_E$ and $\norm{\cdot}_F$ respectively. Whenever the space is understood, we will simply write $\norm{\cdot}$. The open unit ball of $E$ will be denote by $B_E$ and its dual space is $E'$. A pointed metric space is a metric space with a distinguished point, that we will always denote by $0$. In particular, a normed vector space is a pointed metric space and its distinguished point will be $0$. As usual, the space of all linear continuous, compact, approximable and finite rank  operators between $E$ and $F$ will be denoted by $\L(E;F), \K(E;F), \overF(E;F)$ and $\F(E;F)$ respectively. Also, for $1\leq p<\infty$, $\N_p$, $\I_p$ and $\Pi_p$ stands for the ideals of $p$-nuclear, $p$-integral and $p$-summing operators.

Recall that for two metric spaces $X$ and $Y$, a map $f\colon X\rightarrow Y$ is said to be a \textit{Lipschitz map} if there exists a constant $C>0$ such that $d(f(x);f(x'))\leq C d(x,x')$ for all $x, x' \in X$. The least of such constants will be denoted as $\Lip(f)$. We denote by $\Lip_0(X;Y)$ the set of all Lipschitz operators that vanish at $0$. If we consider a Banach space $E$, the space $\big(\Lip_0(X;E);\Lip(\cdot)\big)$ becomes a Banach space. In the case when $E=\mathds K$ ($\mathds K=\mathds R$ or $\mathds C$), we write $\Lip_0(X;\mathbb K)=X^\#$ and it will be called as the \textit{Lipschitz dual} of $X$. We refer to the book of Weaver \cite{Wea99} for more about this space.
Whenever $X$ is a complete pointed metric space, $X^\#$ is a dual Banach space \cite{AE56}. The predual of $X^\#$ it is called the Arens-Eells space (or the Lipschitz free space over $X$), and it will be denoted as $\Ae(X)$.  If we consider the Dirac map $\delta_X\colon X \rightarrow (X^\#)'$ defined as $\delta_X(x)(f)=f(x)$,
then we have the equality
$$
\Ae(X)=\overline{span\big\{\delta_X(x) \colon x \in X\big\}}\subset (X^\#)'.
$$

For $f\in\Lip_0(X;Y)$, there exists a (unique) linear operator $\widehat f\in\L(\Ae(X);\Ae(Y))$ such that $\delta_Y\circ f=\widehat f\circ \delta_X$ (see e.g. \cite[]{Wea99}). When $E$ is a Banach space, there is a linear quotient map $\beta_E \colon \Ae(E)\rightarrow E$, called the \textit{the barycenter}, which is the left inverse of $\delta_E$, meaning $\beta_E \circ \delta_E=Id_E$. There is an isometric isomorphism between $\Lip_0(X;E)$ and $\L(\Ae(X);E)$. Indeed, given $f\in \Lip_0(X;E)$, there exists a unique linear operator $L_f \in \L(\Ae(X);E)$ such that $f=L_f\circ \delta_X$ with $\norm{f}_{\Lip_0}=\norm{L_f}$ (see e.g. \cite[Proposition~2.2.4]{Wea99}). The linear operator $L_f$ is defined as $L_f=\beta_E\circ\widehat f$ ans satisfies $L_f(\delta_{X}(x))=f(x)$. We will refer the operator $L_f$ as \textit{the linearization} of $f$. Note that when $E$ and $F$ are Banach spaces then $\L(E;F) \subset \Lip_0(E;F)$ and if $T\in \L(E;F)$, the linerization of $T$ is $L_T=T\circ \beta_E$ \cite[Lemma~2.5]{GK03}.

By a \textit{Banach Lipschitz operator ideal} we mean a subclass $\I_{\Lip}$ of $\Lip_0$ such that for every pointed metric space $X$ and every Banach space $E$, the component
$$
\I_{\Lip}(X;E)=\Lip_0(X;E)\cap \I_{\Lip}
$$
satisfy:
\begin{enumerate}[\upshape (i)]
\item $\I_{\Lip}(X;E)$ is a linear subspace of $\Lip_0(X;E)$.
\item $Id_\mathds{K} \in \I_{\Lip}(\mathds{K};\mathds{K})$.
\item The ideal property: if $g \in \Lip_0(Y;X), f\in \I_{\Lip}(X;E)$ and $S \in \L(E;F)$, then the composition $S\circ f\circ g \in \I_{\Lip}(Y;F)$.
\end{enumerate}
Also there is the {\it Lipschitz ideal norm} over $\I_{\Lip}$, given by a function $\norm{\cdot}_{\I_{\Lip}}\colon \I_{\Lip}\rightarrow [0,+\infty)$ that satisfies

\begin{enumerate}[\upshape (i')]
\item For every pointed metric space $X$ and every Banach space $E$, the pair $\big(\I_{\Lip}(X;E);\norm{\cdot}_{\I_{\Lip}}\big)$ is a Banach space and $\Lip(f)\leq \norm{f}_{\I_{\Lip}}$ for all $f \in \I_{\Lip}(X;E)$.
\item $\norm{Id_{\mathds K} \colon \mathds K\rightarrow \mathds K}_{\I_{\Lip}}=1$
\item If $g \in \Lip_0(Y;X), f\in \I_{\Lip}(X;E)$ and $S \in \L(E;F)$, then $\norm{S\circ f\circ g}_{\I_{\Lip}} \leq \Lip(g) \norm{f}_{\I_{\Lip}} \norm{S}$.
\end{enumerate}

This definition was introduce in \cite[Definition~2.1]{ARSY16} and independently in \cite[Definition~2.3]{CCJV16}, under the name of {\it generic Lipschitz operator Banach ideal} and extends the definition of Banach linear operator ideals (see e.g. \cite{Pie80}).
Along the manuscript, to avoid confusion, we will denote $\I$, $\J$ for Banach Lipschitz operator ideals while $\A$, $\B$ stands for Banach linear operator ideals. With $\A\subset \B$  we mean that for all Banach spaces $E$ and $F$, $\A(E;F)\subset \B(E;F)$. The same applies with Lipschitz operator ideals.

All other relevant terminology and preliminaries are given in corresponding sections.

\section{Galois Conection between Linear and Lipschitz operator ideals}\label{Sec: Galois}

Following \cite[Section 3]{ARSY16}, given a Banach linear operator ideal $\A$, there is a method to produce a Banach Lipschitz operator ideal. Namely, for a pointed metric space $X$ and a Banach space $E$, a Lipschitz mapping $f$ from $X$ to $E$ belongs to $\A \circ \Lip_0(X;E)$ if there exist a Banach space $F$, a Lipschitz operator $g\in \Lip_0(X;F)$ and a linear operator $T \in \A(F;E)$ such that $f=T\circ g$. Equivalently, by \cite[Proposition~3.2]{ARSY16}, $f \in \A \circ \Lip_0(X;E)$ if and only if its linearization $L_f$ belongs to $\A(\Ae(X);E)$. The norm in $\A \circ \Lip_0$ is defined as,
$$
\norm{f}_{\A\circ \Lip_0}=\norm{L_f}_{\A}.
$$
The Banach Lipschitz operator ideals which are obtained in this way are called {\it of composition type}. For example the finite rank, approximable and compact linear operator ideals produce the finite rank, approximable and compact Banach Lipschitz operator ideals respectively, which were introduced in \cite{JiSeVi14} and denoted by $\Lip_{0\F}, \Lip_{0\overline\F}$ and $\Lip_{0\K}$. Also, for $1\leq p < \infty$ the Banach Lipschitz ideal of {\it strongly $p$-nuclear} operators, $\N^{SL}_p$, introduced in \cite{CheZhe12} and the Banach Lipschitz ideal of {\it strongly $p$-integral} operators, $\I^{SL}_p$, introduced in \cite{JiSeVi14} are obtained via the $p$-nuclear and $p$-integral linear operator ideals respectively \cite[Proposition~2.9 and Proposition~2.11]{Saa17}.

On the other hand, Banach spaces are pointed metric spaces (with distinguished point $0$) and linear operators are Lipschitz operators. This observation allow us to produce a Banach linear operator ideal from a Banach Lipschitz operator ideal, as it was noticed in \cite[Remark~2.9]{ARSY16}. If $\I$ is a Banach Lipschitz operator ideal, then the Banach linear operator ideal $\I \cap \L$ is defined, for Banach spaces $E$ and $F$ as
$$
\I \cap \L(E;F)=\Big\{T \in \L(E;F) \ \colon  \ T\in\I(E;F) \Big\}
$$
with the norm
$$
\norm{T}_{\I\cap \L}=\norm{T}_{\I}.
$$

Combining this two methods, we obtain a {\it procedure} (in the sense of Piestch) for Banach linear operator ideals.

\begin{proposition}\label{Prop: ban} Let $\A$ be a Banach linear operator ideal. Then
$$
\A \subset (\A\circ \Lip_0)\cap \L.
$$
Moreover, for any Banach spaces $E$ and $F$ and $T \in \A(E;F)$ we have
$$
\norm{T}_{(\A \circ \Lip_0)\cap \L}\leq \norm{T}_\A.
$$
\end{proposition}

\begin{proof}
Take $E$ and $F$ Banach spaces and let $T \in \A(E;F)$. Since $T$ is linear, it is enough to show that its linearization belongs to $\A$. This follows since $T \circ \beta_E =L_T$.  Moreover,
$$
\norm{T}_{(\A \circ \Lip_0)\cap \L}=\norm{T}_{\A \circ \Lip_0}=\norm{L_T}_\A= \norm{T \circ \beta_E}_\A\leq \norm{T}_\A,
$$
and the proof finish.
\end{proof}

The equality in Proposition~\ref{Prop: ban} holds whenever we consider linear operators from Banach spaces which have the Lipschitz-lifting property defined by Godefroy and Kalton. Following \cite[Definition~2.7]{GK03} a Banach space $E$ has the \textit{Lipschitz-lifting property} if there exists a linear operator $U\colon E\rightarrow \Ae(E)$ such that $\beta_E\circ U=Id_E$. If the linear operator $U$ has norm one, then the Banach space $E$ has the \textit{isometric Lipschitz-lifting property}. For example, all separable Banach spaces and every Lipschitz free-space over any Banach space have the isometric Lipschitz-lifting property \cite[Lemma~2.10 and Theorem~3.1]{GK03}.

\begin{proposition}\label{Prop:lifting property}
Let $\A$ be a Banach linear operator ideal and $E$, $F$ be Banach spaces. If $E$ has the isometric Lipschitz-lifting property then
$$
\A(E;F) = (\A\circ \Lip_0)\cap \L(E;F) \quad isometrically.
$$
\end{proposition}
\begin{proof}
For $T\in \L(E;F)$, we have the equality $L_T=T\circ \beta_E$. Since $E$ has the isometric Lipschitz-lifting property, we have the equality $T=T\circ \beta_E \circ U=L_T \circ U$.
Thus, if $T\in (\A\circ \Lip_0)\cap \L(E;F)$, then $L_T \in \A(\Ae(E);F)$, which implies that $T\in \A(E;F)$ with $\|T\|_\A=\|L_T\circ U\|_\A\leq \|L_T\|_\A=\|T\|_{(\A\circ \Lip_0)\cap \L}$. The proof finish.
\end{proof}

The particular case of the above proposition when $\A$ is the ideal of $p$-nuclear operators covers \cite[Theorem~2.3]{CheZhe12}.

Also, for some Banach linear operator ideals, the equality in Proposition~\ref{Prop: ban} holds in general, regarding the Banach spaces we consider. This is the case when the Banach linear operator ideal is maximal or surjective. Following \cite[17.2]{DF93}, we may consider the maximal hull of a Banach linear operator ideal $\A$, $\A^{max}$, as the {\it biggest} Banach linear operator ideal which coincide with $\A$ over finite dimensional spaces. Moreover, $\A$ is said to be maximal if $\A=\A^{max}$ isometrically. For instance, the ideals of $p$-summing operators and $p$-integral operators are maximal, $1\leq p<\infty$ \cite[19.1.1 and 23.1.3]{Pie80}.

\begin{proposition} Let $\A$ be a maximal Banach linear operator ideal. Then
$$
\A = (\A\circ \Lip_0)\cap \L \quad isometrically.
$$
\end{proposition}
\begin{proof}
Since every finite dimensional space has the isometric Lipschitz-lifting property, by Proposition~\ref{Prop:lifting property}, we have the isometric equality $\A(M;N) = (\A\circ \Lip_0)\cap \L(M;N)$ for every finite dimensional spaces $M$ and $N$. As $\A$ being maximal, by \cite[Remark 17.2]{DF93} we have the continuous inclusion $(\A\circ \Lip_0)\cap \L \subset \A$. The other inclusion always holds by Proposition~\ref{Prop: ban}.
\end{proof}

 Recall that a Banach operator ideal $\A$ is \textit{surjective} if for every Banach spaces $E$ and $F$, a linear operator $T\in \L(E;F)$ belongs to $\A(E;F)$ whenever exist a Banach space $G$ and a linear operator $R\in \A(G;F)$ such that $T(B_E)\subset R(B_G)$. Moreover, $\|T\|_{\A}=\inf\{\|R\|_\A \colon R\in \A(G;F), \ T(B_E)\subset R(B_G)\}$. For instance, the ideal of compact operators and weakly compact operators are surjective ideals.

\begin{proposition}
  Let $\A$ be a surjective Banach linear operator ideal. Then
$$
\A = (\A\circ \Lip_0)\cap \L \quad isometrically.
$$
\end{proposition}
\begin{proof}
Fix $E$ and $F$ Banach spaces and take $T\in (\A\circ \Lip_0)\cap \L(E;F)$. Then, $L_T \in \A(\Ae(E);F)$. From the proof of \cite[Proposition~2.1]{JiSeVi14} we may deduce the equality $L_T(B_{\Ae(E)})=T(B_E)$. Since $\A$ is surjective, then $T\in \A(E;F)$ and $\|T\|_\A\leq \|L_T\|_\A=\|T\|_{(\A\circ \Lip_0)\cap \L}$.
\end{proof}

The equality in Proposition~\ref{Prop: ban} does not holds in general. At the light of Proposition~\ref{Prop:lifting property}, to construct a counterexample we will appeal to Banach spaces with the Lipschitz lifting property.

\begin{definition} Let $E$ and $F$ be Banach spaces and $T\in \L(E;F)$. Then $T$ is called LLP-factorable if there exists a Banach space $G$ with the isometric Lipschitz lifting property and operators $S\in \L(E;G)$ and $R\in \L(G;F)$ such that $T=R\circ S$. The class of all LLP-factorable operators is denoted by $\mathcal{LP}$.
\end{definition}

We think that the next lemma is well known. However, we did not find it in the literature.
\begin{lemma}\label{Lemma: Sum Llp}
Let $(E_i)_{i \in \mathds N}$ a sequence of Banach spaces with the isometric Lipschitz-lifting property. Then the Banach space
$$
E=\{e=(e_1,e_2,\ldots) \colon e_i \in E_i \quad \|e\|_E=\sum_{i=1}^{\infty}\|e_i\|_{E_i}<\infty\}
$$
endowed with the norm $\|\cdot\|_E$ has the isometric Lipschitz-lifting property.
\end{lemma}
\begin{proof}
Consider the Banach space
$$
U=\{u=(u_1,u_2,\ldots) \colon u_i \in \Ae(E_i) \quad \|u\|_U=\sum_{i=1}^{\infty}\|u_i\|_{\Ae(E_i)}<\infty\}
$$
endowed with the norm $\|\cdot\|_U$. Note that
$$
U=\overline{span\{(0,\ldots,0,\delta_{E_i}(e_i),0,\ldots)\colon e_i \in E_i, i \in \mathds N\}}^{\|\cdot\|_U}.
$$
Since every Banach space $E_i$, for $i\in\mathds{N}$, has the isometric lifting Lipschitz property, there exist linear operators $T_i\colon E_i\rightarrow \Ae(E_i)$ such that $\beta_{E_i}\circ T_i=Id_{E_i}$. For every $i \in \mathds N$ and $e_i \in E_i $ there exist sequences $(\widetilde{e}_{i,j})_j \subset E_i$ and $(\alpha_{i,j})_j\subset\mathds{K}$ with $\sum_{j=1}^{\infty} |\alpha_{i,j}| \|\widetilde{e}_{i,j}\|_{E_i}<\infty$ such that $T_i(e_i)=\sum_{j=1}^{\infty} \alpha_{i,j} \delta_{E_i}(\widetilde{e}_{i,j})$. Note that
$$e_i=\beta_{E_i}\circ T_i(e_i)=\beta_{E_i} \left(\sum_{j=1}^{\infty} \alpha_{i,j} \delta_{E_i}(\widetilde{e}_{i,j})\right)=\sum_{j=1}^{\infty} \alpha_{i,j} \widetilde{e}_{i,j}
$$
Define the linear operators
$T\colon E\rightarrow U$ and $\xi\colon U \rightarrow \Ae(E)$ as
$$
T(e_1,e_2,\ldots)=(T_1e_1,T_2e_2,\ldots)
$$
and
$$
\xi(0,\ldots,0,\delta_{E_i}(e_i),0,\ldots)=\delta_E(0,\ldots,0,e_i,0,\ldots),
$$
both extended by linearity and continuity.

Is clear that $\|\xi \circ T\|=1$. Finally, we show that $\beta_E\circ(\xi\circ T)=Id_E$. Indeed,
$$
\begin{array}{rl}
\d\beta_E\circ \xi \circ T(e_1,e_2,\ldots)=&\d\beta_E\circ \xi \left(T_1(e_1),T_2(e_2),\ldots\right)\\
=&\d\sum_{i=1}^{\infty}\beta_E\circ \xi (0,\ldots,0,T_i(e_i),0,\ldots)\\
=&\d\sum_{i=1}^{\infty}\beta_E\circ \xi\left(\sum_{j=1}^{\infty} (0,\ldots,0,\alpha_{i,j} \delta_{E_i}(\widetilde{e}_{i,j}),0,\ldots)\right)\\
=&\d\sum_{i=1}^{\infty}\sum_{j=1}^{\infty} \alpha_{i,j}\beta_E\circ \xi(0,\ldots,0, \delta_{E_i}(\widetilde{e}_{i,j}),0,\ldots)\\
=&\d\sum_{i=1}^{\infty}\sum_{j=1}^{\infty} \alpha_{i,j}\beta_E\circ \delta_E(0,\ldots,0, \widetilde{e}_{i,j},0,\ldots)\\
=&\d\sum_{i=1}^{\infty}\sum_{j=1}^{\infty} \alpha_{i,j} (0,\ldots,0, \widetilde{e}_{i,j},0,\ldots)\\
=&\d\sum_{i=1}^{\infty}(0,\ldots,0, e_i,0,\ldots)\\
=&\d(e_1,e_2,\ldots)
\end{array}
$$
\end{proof}
The above lemma is the key to show that $\mathcal{LP}$ is a Banach linear operator ideal. The proof is  straightforward and we will omit it.

\begin{theorem} The class $\mathcal {LP}$ endowed with the ideal norm defined, for every Banach spaces $E$ and $F$ and $T\in \mathcal {LP}(E,F)$, as
$$
\|T\|_{\mathcal {LP}}=\inf\{\|R\| \|S\|\}
$$
where the infimum is taken over all the factorization of $T=R\circ S$ trough a Banach space with the Lipschitz-lifting property, is a Banach linear operator ideal.
\end{theorem}

\begin{theorem}  The inclusion $\mathcal{LP} \subset (\mathcal{LP} \circ \Lip_0)\cap \L$ is strict.
\end{theorem}
\begin{proof}
Take $E$ a Banach space without the Lipschitz-lifting property (for example $E=c_0(\Gamma)$ with $\Gamma$ an uncountable set, see \cite[p. 128]{GK03}), and consider the identity map $Id_E$. Since, by \cite[Lemma~2.10]{GK03}, $\Ae(E)$ has the isometric Lipschitz-lifting property, the linearization of $Id_E$, $L_{Id_E}\in \mathcal{LP}(\Ae(E),E)$. This implies that $Id_E \in (\mathcal{LP} \circ \Lip_0)\cap \L(E,E)$.

Now, suppose that $Id_E \in \mathcal{LP}(E,E)$. Then there exist a Banach space $G$ with the isometric Lipschitz-lifting property and operators $S\in \L(E;G)$ and $R\in \L(G;E)$ such that $Id_E=R\circ S$. Consider the following commutative diagram
$$
\xymatrix{
E\ar[rr]^{Id_E} \ar[dr]_{S} & & E\ar[dd]^{\delta_E}\\
& G \ar[d]_{\delta_G} \ar[ur]_{R} &\\
& \Ae(G)\ar[uur]_{L_R}\ar[r]_{\widehat{R}}& \Ae(E)
}.
$$
Also, as $G$ has the isometric Lipschitz-lifting property, then there exists an operator $U\in \L(G,\Ae(G))$ such that $\beta_G\circ U=Id_G$. Take $T=\widehat{R}\circ U\circ S \in \L(E;\Ae(E))$. Then $\beta_E \circ T=L_{R}\circ U\circ S$. As $R\in\L(G;E)$, then $L_R=R\circ \beta_G$, then
$$
\beta_E \circ T= R\circ \beta_G\circ U\circ S=R\circ S=Id_E.
$$
Then, $E$ has the Lipschitz-lifting property which is a contradiction.
\end{proof}

Now we focus on Banach Lipschitz operator ideals. In the same way we did for Banach linear operator ideals, we may proceed to obtain a procedure for Banach Lipschitz operator ideals.

\begin{proposition}\label{Prop: lip} Let $\I$ be a Banach Lipschitz operator ideal. Then
$$
(\I \cap \L) \circ \Lip_0 \subset \I.
$$
Moreover, for any pointed metric space $X$, any Banach space $E$ and $f\in (\I \cap \L) \circ \Lip_0 (X;E)$ we have
$$
\norm{f}_{\I}\leq \norm{f}_{(\I \cap \L) \circ \Lip_0}.
$$
\end{proposition}
\begin{proof} Let $X$ be a pointed metric space and $E$ be a Banach space. A Lipschitz operator $f \in (\I \cap \L) \circ \Lip_0(X;E)$ if and only if $L_f \in (\I \cap \L)(\Ae(X);E)$. Then $L_f \in \I(\Ae(X);E)$ and, since $f=L_f \circ \delta_X$, we conclude that $f \in \I(X;E)$. Moreover,
$$
\norm{f}_{\I}=\norm{L_f \circ \delta_X}_{\I}\leq \norm{L_f}_{\I}=\norm{L_f}_{\I \cap \L}=\norm{f}_{(\I\cap\L)\circ \Lip_0},
$$
and the proof finish.
\end{proof}

This methods to obtain Banach Lipschitz operator ideals from a Banach linear operator ideal and viceversa fits in the theory of {\it Galois connection between lattices}. Recall that a non-empty ordered set $P$ is a complete lattice if for every subset $S\subset P$, $\bigvee S=\sup S$ and $\bigwedge S=\inf S$ belongs to $P$. We consider the lattices $\LIP$ and $\LIN$ of all Banach Lipschitz operator ideals and Banach linear operator ideals, respectively, both consider with the order given by the continuous inclusion. That is, for $\I,\J$ in $\LIP$, $\I\leq \J$ if and only if $\I\subset \J$ and, for every pointed metric space $X$ and every Banach space $E$, if $f\in \I(X;E)$, then $\|f\|_{\J}\leq\|f\|_{\I}$.
It is well known that in $\LIN$ the supremum is $\L$, meanwhile the infimum is the Banach ideal of nuclear operators, $\N$. For the case of $\LIP$, we have that the supremum is $\Lip_0$ and the infimum is $\N\circ\Lip_0$ as the next lemma shows. Following \cite[Proposition~4.2]{CaJi15}, a Lipschitz operator $f\in \N\circ \Lip_0(X;E)$ if there exist sequences $(f_n)_n\subset X^\#$ y $(e_n)_n\subset E$ with $\sum_n^{\infty} \Lip(f_n) \|e_n\|<\infty$ such that  $f=\sum_n^{\infty} f_n e_n$. Moreover $\|f\|_{\N \circ \Lip_0}=\inf\{\sum_n^\infty \Lip(f_n) \|e_n\|\}$, where the infimum is taken over all the representations of $f$.

\begin{lemma}\label{Lemma: Nuclear Min}
Let $\I$ be a Banach Lipschitz operator ideal. Then $\N\circ \Lip_0\subset \I$. Moreover, for all pointed metric space $X$ and Banach space $E$, if $f\in \N\circ \Lip_0(X;E)$, then $\|f\|_{\I}\leq \|f\|_{\N\circ \Lip_0}$.
\end{lemma}
\begin{proof}
As it was shown in \cite[Proposition~4.2]{CaJi15}, every $f\in \N\circ \Lip_0(X;E)$ is the  $\|\cdot\|_{\N\circ \Lip_0}$-limit of Lipschitz finite rank operators. The claim will follows if we show that for every $g\in \Lip_{0\F}(X;E)$, $\|g\|_{\I}\leq\|g\|_{\N\circ \Lip_0}$. As we can derive from  \cite[Corollary~2.6 (iii)]{JiSeVi14} we may take $g=\sum_{j=1}^{n} g_j e_j$ where $g_j\in X^{\#}$ and $e_j \in E$, $j=1,\ldots,n.$ Then
$$
\|g\|_{\I}=\Big\|\sum_{j=1}^{n} g_j e_j\Big\|_{\I}\leq \sum_{j=1}^{n} \|g_j e_j\|_{\I}=\sum_{j=1}^{n} \|g_j e_j\| =\sum_{j=1}^{n} \Lip(g_j)\|e_j\|,
$$
and taking the infimum over all the representations of $g$, we get that $\|g\|_{\I}\leq \|g\|_{\N\circ \Lip_0}$. The proof follows.
\end{proof}

To see that $\LIN$ is a complete lattice, we will construct the infimum of a subset of $\LIN$ in the next well known lemma. The same construction can be applied to $\LIP$. We omit the proof.

\begin{lemma}\label{Lemma: infimum}
\begin{enumerate}[\upshape (a)]
\item For any non-empty subset $S \subset \LIN$, consider $\bigcap S$ defined, for Banach spaces $E$ and $F$ as
$$
\bigcap S(E;F)=\Big\{T\in \L(E;F) \colon T \in \A(E;F) \quad \forall \ \A \in S \ \mbox{with} \ \sup_{\A \in S} \|T\|_{\A}<\infty\Big\},
$$
with the norm $\|T\|_{\bigcap S} = \sup_{\A \in S} \|T\|_{\A}$.
Then $(\bigcap S; \|\cdot\|_{\bigcap S})$ is a Banach linear operator ideal such that
$\bigcap S\subset \A$ for every $\A \in S$.

Moreover, if there exists a Banach operator ideal $\B$ such that $\B\subset \A$ for every $\A \in S$, then $\B \subset \bigcap S$ and, for every Banach spaces $E$ and $F$ and $T\in\B(E,F)$, $\|T\|_{\bigcap S}\leq \|T\|_{\B}$.

\item For any non-empty subset $S \subset \LIP$, consider $\bigcap S$ defined, for a pointed metric space $X$ and a Banach space $E$ as
$$
\bigcap S(X;E)=\Big\{f\in \Lip_0(X;E) \colon f \in \I(X;E) \quad \forall \ \I \in S \ \mbox{with} \ \sup_{\I \in S} \|T\|_{\I}<\infty\Big\},
$$
with the norm $\|f\|_{\bigcap S} = \sup_{\I \in S} \|f\|_{\I}$.
Then $(\bigcap S; \|\cdot\|_{\bigcap S})$ is a Banach Lipschitz operator ideal such that
$\bigcap S\subset \I$ for every $\I \in S$.

Moreover, if there exists a Banach Lipschitz operator ideal $\J$ such that $\J\subset \I$ for every $\I \in S$, then $\J \subset \bigcap S$ and, for every pointed metric space $X$ and every Banach space $E$ and $f\in\J(X,E)$, $\|f\|_{\bigcap S}\leq \|f\|_{\J}$.
\end{enumerate}
\end{lemma}

\begin{proposition}\label{Prop: Complete lattice}
$\LIN$ and $\LIP$ are complete lattices with the order given by the continuous inclusion.
\end{proposition}
\begin{proof} Since $\LIP$ has top element and, by the above lemma,  every non-empty subset of $\LIP$ has infimum, an application of \cite[Theorem~2.31]{DP02} gives that $\LIP$ is a complete lattice. The same holds for $\LIN$.
\end{proof}

Now we are ready to establish the main result of this section, which states the relationship between Banach Lipschitz operator ideals and Banach linear operator ideals. Recall that for $P$ and $Q$ ordered sets, a pair $(\psi,\phi)$ of maps $\psi:P\to Q$ and $\phi:Q\to P$ is a \textit{Galois connection} between $P$ and $Q$ if for all $p\in P$ and $q\in Q$ holds that $\psi(p)\leq q$ if and only if $p\leq\phi(q)$.

\begin{theorem}\label{Teo: Galois} The maps $\psi\colon \LIN \rightarrow \LIP$ and $\phi\colon \LIP\rightarrow\LIN$ defined as
$$
\psi(\A)=\A\circ \Lip_0 \quad \mbox{and} \quad\phi(\I)=\I\cap \L,
$$
set up a Galois connection between $\LIN$ and $\LIP$.
\end{theorem}
\begin{proof}
It is clear that if $\A,\B \in \LIN$ with $\A\leq \B$, then $\psi(\A) \leq \psi(B)$ and that if $\I,\J \in \LIP$ with $\I\leq \J$, then $\phi(\I) \leq \phi(\J)$. Then the result follows by \cite[Lemma~7.26]{DP02} in combination with Proposition~\ref{Prop: ban} and Proposition~\ref{Prop: lip}.
\end{proof}

Also, as a direct consequence of \cite[Lemma~7.26]{DP02} we have the following results.
\begin{corollary}\label{Coro: ban}
Let $\A$ be a Banach linear operator ideal. Then
$$
\left( (\A\circ \Lip_0)\cap\L\right)\circ \Lip_0=\A\circ \Lip_0 \quad isometrically,
$$
\end{corollary}
\begin{corollary}\label{Coro: lip}
Let $\I$ be a Banach Lipschitz operator ideal. Then
$$
\left((\I\cap\L)\circ \Lip_0\right)\cap\L=\I\cap\L \quad isometrically.
$$
\end{corollary}

Next we formulate a criterion which allow us to identify whenever a Lipschitz Banach operator ideal is of composition type or not.

\begin{proposition}[Criterion]\label{Prop: Criterion} Let $\I$ be a Banach Lipschitz operator ideal. Then $\I$ is of composition type if and only if
$$
\I=(\I\cap \L)\circ \Lip_0 \quad isometrically.
$$
\end{proposition}
\begin{proof}
  The {\it only if} part is trivial. Now, if $\I=\A\circ\Lip_0$ for some Banach linear operator $\A$, then an application of Corollary \ref{Coro: ban} gives the proof.
\end{proof}

To finish this section, we apply the criterion to show that the Banach Lipschitz operator ideals of $p$-summing operators, $p$-integral operators (both introduced by Farmer and Johnson \cite{FaJo09}) and $p$-nuclear operators (introduced by Chen and Zheng \cite{CheZhe12}) are not of composition type. First, we briefly describe them. Fix $1\leq p <\infty$, a pointed metric space $X$, a Banach space $E$ and a Lipschitz function $f \in \Lip_0(X;E)$.

The function $f$ is a Lipschitz $p$-summing operator if there exists a constant $C>0$ so that for all choices of $(x_i)_{i=1}^{n}$ and $(y_i)_{i=1}^{n}$ in $X$ the inequality
$$
\sum_{i=1}^{n} \norm{f(x_i)-f(y_i)}^p \leq C^{p} \sup_{f^{\#} \in B_{X^\#}} \sum_{i=1}^n|f^{\#}(x_i)-f^{\#}(y_i)|^p
$$
holds. The space of all Lipschitz $p$-summing operators from $X$ to $E$ is denoted by $\Pi_p^L(X;E)$ and the least of such constants $C$ is the $p$-summing norm of $f$ and it is denoted as $\pi_p^{L}(f)$.

The function $f$ is a Lipschitz $p$-integral operator if there are a probability measure space $(\Omega,\Sigma,\mu)$ and two Lipschitz mappings $A\colon L_p(\mu)\rightarrow E''$ and $B\colon X\rightarrow L_{\infty}(\mu)$ such that the following diagram commutes:
\begin{equation}\label{Ec 1}
\xymatrix{
X\ar[d]_B\ar[r]^{f} & E \ar[r]^{J_E} & E'' \\
L_{\infty}(\mu) \ar[rr]^{i_p}& &L_p(\mu)\ar[u]_{A}
}
\end{equation}
where $J_E\colon E\rightarrow E''$ is the canonical evaluation map and $i_p\colon L_{\infty}(\mu)\rightarrow L_p(\mu)$ is the inclusion map. The space of all Lipschitz $p$-integral operators from $X$ to $E$ is denoted by $\I_p^L(X;E)$ and the Lipschitz $p$-integral norm of $f$ is the infimum of $\Lip(A) \Lip(B)$ taken over all factorizations of $f$ as in \eqref{Ec 1} and it is denoted by $\iota_p^{L}(f)$.

Finally, the function $f$ is a Lipschitz $p$-nuclear operator if there are two Lipschitz mappings $B\colon X\rightarrow \ell_{\infty}$ and $A\colon \ell_p\rightarrow E$ and a sequence $\lambda=(\lambda_n)_n \in \ell_p$ such that the following diagram commutes:
\begin{equation}\label{Ec 2}
\xymatrix{
X\ar[d]_B\ar[r]^{f} & E \\
\ell_{\infty}\ar[r]^{M_{\lambda}} & \ell_p\ar[u]_{A}
},
\end{equation}
where $M_{\lambda}$ is the diagonal operator. The space of all Lipschitz $p$-nuclear operators from $X$ to $E$ is denoted by $\N^{L}_p(X;E)$ and the infimum of $\Lip(A) \|\lambda\|_{\ell_p}\Lip(B)$ taken over all the factorization of $f$ as in \eqref{Ec 2} is the $p$-nuclear norm and is denoted by $\nu_p^{L}(f)$.

In \cite{ARSY16} it was shown that all this three spaces, $(\Pi_p^{L};\pi_p^{L}), (\I_p^{L};\iota_p^{L})$ and $(\N_p^{L};\nu_p^{L})$ are Banach Lipschitz operators ideals. Moreover, we have the inclusions $\N_p^{L}\subset \I_p^{L} \subset \Pi_p^{L}$. As it was noticed in \cite[Remark~3.7]{ARSY16}, the Dirac map $\delta_{\mathds R}\colon \mathds R \rightarrow \Ae(\mathds R)$ belongs to $\mathcal N_p^{L}(\mathds R;\Ae(\mathds R))$. Hence it is also Lipschitz $p$-integral and Lipschitz $p$-summing (for this last see also \cite[Remark~3.3]{Saa15}). Besides, regarding the operator $\delta_{\mathds R}$ we have the following proposition.
\begin{proposition}\label{Prop: Delta} Let $\A$ be a Banach operator ideal. Then $\delta_{\mathds R} \in \A\circ \Lip_0(\mathds R;\Ae(\mathds R))$ if and only if  $Id_{L_1} \in \A(L_1;L_1)$.
\end{proposition}
\begin{proof}
First recall that the Arens-Ells space of $\mathds R$ is isometrically isomorphic to $L_1$ (see e.g. \cite{God15}). The result follows since the linearization of $\delta_{\mathds R}$ is the identity map of $L_1$.
\end{proof}

\begin{proposition}\label{Prop: No compo}
For $1\leq p <\infty$, the Banach Lipschitz operator ideals $\Pi_p^L$, $\I_p^L$ and $\N_p^L$ are not of composition type.
\end{proposition}
\begin{proof}
First suppose that $\Pi_p^{L}$ is of composition type. Applying Proposition~\ref{Prop: Criterion}, we get $\Pi_p^{L}(\mathds R;\Ae(\mathds R))= (\Pi_p^L\cap\L)\circ \Lip_0 (\mathds R;\Ae(\mathds R))$. Since $\Pi_p^L\cap\L(\Ae(\mathds R);\Ae(\mathds R))=\Pi_p(\Ae(\mathds R);\Ae(\mathds R))$ \cite[Theorem~2]{FaJo09}, we obtain the equality $\Pi_p^{L}(\mathds R;\Ae(\mathds R))=\Pi_p \circ \Lip_0 (\mathds R;\Ae(\mathds R))$. As the Dirac map $\delta_{\mathds R} \in \Pi_p^{L}(\mathds R;\Ae(\mathds R))$, using Proposition~\ref{Prop: Delta}, we obtain that $Id_{L_1}$ is a $p$-summing linear operator, arriving to a contradiction.

The cases of Lipschitz $p$-integral and Lipschitz $p$-nuclear operators follows in the same way
using the fact that $\I_p^{L}\cap \L=\I_p$ (see below \cite[Problem 1]{FaJo09}) and $\N_p^{L}\cap \L(\Ae(\mathds R);\Ae(\mathds R)) =\N_p(\Ae(\mathds R);\Ae(\mathds R))$ \cite[Theorem~2.1]{CheZhe12}.
\end{proof}

%\txcr{tiene sentido dar criterio para ver cuando un ideal de operadores viene de intersecar un Lipschitz con lineales? se le pone nombre? Lipschitz type?}

\section{Minimal Banach Lipschitz operator ideals}\label{Sec: minimal}

Let us return to our approach of $\LIN$ and $\LIP$ as lattices and their Galois connection.
Following \cite[17.2 and 22.1]{DF93}, for a Banach linear operator $\A$, its minimal kernel $\A^{min}$ and its maximal hull $\A^{max}$ are the {\it smallest} and the {\it biggest} Banach linear operator ideals such that coincide with $\A$ over all finite dimensional Banach spaces. From now on, we denote with $\FIN$ to the class of Banach spaces of finite dimension. In the framework of lattices, the minimal kernel and the maximal hull of $\A$  can be seen as the infimum and the supremum of the sublattice of $\LIN$
$$
\CLIN_{\A}=\big\{\B \in \LIN \colon \B(M;N)=\A(M;N) \ \forall \ M,N \in \FIN\big\}.
$$
That is, $\bigvee \CLIN_{\A}=\A^{max}$ and  $\bigwedge \CLIN_{\A}=\A^{min}$.

The maximal hull of a Banach Lipschitz operator ideal was introduce by Cabrera--Padilla, Ch\'avez--Dominguez, Jimenez--Vargas and  Villegas--Vallecillos  in \cite{CCJV16}. From \cite[Lemma~2.4]{CCJV16}, we can deduce that, for a Banach Lipschitz operator ideal $\I$, its maximal hull, $\I^{max}$, coincide with the supremum of the sublattice of $\LIP$,
$$
\CLIP_{\I}=\big\{\J \in \LIP \colon \J(X_0;N)=\I(X_0;N) \ \forall \ X_0 \in \MFIN, \ N \in \FIN\big\},
$$
where $\MFIN$ stands for the class of all finite pointed metric spaces. In other words we have the equality $\I^{max}=\bigvee \CLIP_{\I}$. In particular, for Banach Lipschitz ideals of composition type we can characterize its maximal hull as follows.

\begin{proposition}\label{prop: Comp Maximal}
Let $\A \in \LIN$. Then
$$
(\A\circ \Lip_0)^{max}=\psi (\bigvee \CLIN_{\A})=\psi(\A^{max})=\A^{max}\circ \Lip_0.
$$
\end{proposition}
\begin{proof}
We have to show that if $\I=\A\circ\Lip_0$, then $\I^{max} = \A^{max}\circ \Lip_0$ holds  isometrically. Indeed, from \cite[Corollary 3.3]{Saa17}, we have that $\A^{max}\circ \Lip_0$ is a maximal Banach Lipschitz operator ideal. The proof follows by showing that for every $X_0\in \MFIN$ and $N \in \FIN$ the isometric equality $\A^{max}\circ \Lip_0(X_0;N)=\A\circ \Lip_0(X_0;N)$ holds.
Let $X_0\in\MFIN$ and $N\in\FIN$. Since $\Ae(X_0)\in \FIN$, then $\A^{max}(\Ae(X_0);N)=\A(\Ae(X_0);N)$ isometrically. An application of \cite[Proposition~3.2]{ARSY16} gives the isometric equality $\A^{max}\circ \Lip_0(X_0;N)=\A\circ \Lip_0(X_0;N)$.
\end{proof}

In order to introduce the minimal kernel of a Banach Lipschitz operator ideal, we need the following result.

\begin{proposition}
  Let $\I \in \LIP$. Then $\CLIP_{\I}$ is a complete sublattice of $\LIP$.
\end{proposition}
\begin{proof}
The proof is analogous to that of Proposition~\ref{Prop: Complete lattice}. Note that the top element of $\CLIP_{\I}$ is $\I^{max}$ and that for every subset $S$ of $\CLIP_{\I}$, the Lipschitz operator ideal $\bigcap S$ defined in Lemma~\ref{Lemma: infimum}  satisfies that  $\bigcap S(X_0;N)=\I(X_0;N)$ for every $X_0 \in \MFIN$ and $N \in \FIN$. Then $\bigcap S \in \CLIP_{\I}$ and the proof finish.
\end{proof}

Since $\CLIP_{\I}$ is a complete lattice we can ensure that there exists the bottom element. Now we are ready to introduce the minimal kernel of a Banach Lipschitz operator ideal.
\begin{definition}

Let $\I \in \LIP$. Then, $\I^{min}=\bigwedge \CLIP_{\I}$.
\end{definition}

From the definitions, for a Banach Lipschitz operator ideal $\I$ and for all finite pointed metric space $X_0$ and every finite dimensional space $N$ we have
$$
\I^{min}(X_0;N)=\I(X_0;N)=\I^{max}(X_0;N) \quad isometrically
$$
and, in general
$$
\I^{min}\leq \I \leq \I^{\max}
$$
holds. Moreover, we have
\begin{proposition}\label{prop:maxmin}
Let $\I$ be a Banach Lipschitz operator ideal. Then
$$
(\I^{max})^{min}=\I^{min} \quad \mbox{and} \quad (\I^{min})^{max}=\I^{max} \quad isometrically.
$$
\end{proposition}
Also, note that since $\I^{min}$ is a Banach Lipschitz operator ideal, the finite rank Lipschitz operators belongs to $\I^{min}$. Then, it follows that $\overline{\Lip_{0 \F}}^{\|\cdot\|_{\I^{min}}}\leq \I^{min}$. But, since $\overline{\Lip_{0 \F}}^{\|\cdot\|_{\I^{min}}}(X_0,N)=\I(X_0,N)$ isometrically whenever $X_0$ is a finite pointed metric space and $N$ is a finite dimensional Banach space we have the following result.

\begin{proposition}\label{Imin mas chico} Let $\I$ be a Banach Lipschitz operator ideal. Then
$$
\I^{min}=\overline{\Lip_{0 \F}}^{\|\cdot\|_{\I^{min}}}.
$$
In particular, every Lipschitz operator in $\I^{min}$ is approximable.
\end{proposition}

In order to characterize the minimal kernel of a Banach Lipschitz operator ideal of composition type, we first introduce, for a Banach Lipschitz operator ideal $\I$, the sublattice of $\LIP$
$$
\CLIP_{\I^{+}}=\big\{\J \in \LIP \colon \J(X_0;N)\supset\I(X_0;N) \ \forall \ X_0 \in \MFIN, \ N \in \FIN\big\}.
$$

\begin{proposition}\label{prop: Minimal} Let $\I \in \LIP$. Then $\CLIP_{\I^{+}}$ is a complete sublattice. Moreover, we have $\bigwedge \CLIP_{\I^{+}}=\I^{min}$.
\end{proposition}
\begin{proof}
It is clear that the top element of $\CLIP_{\I^{+}}$ is $\Lip_0$. Also, for a non-empty subset $S$ of $\CLIP_{\I^{+}}$, then the ideal $\bigcap S$ constructed in Lemma~\ref{Lemma: infimum} belongs to $\CLIP_{\I^{+}}$. It follows from \cite[Theorem~2.31]{DP02} that $\CLIP_{\I^{+}}$ is a complete sublattice.

Moreover, since $\CLIP_{\I}\subset \CLIP_{\I^{+}}$, then $\bigwedge\CLIP_{\I^{+}}\leq \bigwedge\CLIP_{\I}=\I^{min}$. On the other hand, as $\bigwedge\CLIP_{\I^{+}} \in\CLIP_{\I^{+}}$, then $\bigwedge\CLIP_{\I^{+}}(X_0,N)\supset\I(X_0,N)$ for all $X_0$ a finite pointed metric space and $N$ a finite dimensional Banach space. Also, note that $\I$ belongs to $\CLIP_{\I^{+}}$, then $\bigwedge\CLIP_{\I^{+}}\leq\I$, implying that $\bigwedge\CLIP_{\I^{+}}\in \CLIP_{\I}$ and, as a consequence, $\I^{min}=\bigwedge\CLIP_{\I}\leq \bigwedge\CLIP_{\I^{+}}$ and the proof finish.
\end{proof}

As a direct consequence of the above, we have
\begin{proposition}\label{prop: inclusion minimal} Let $\I$ and $\J$ be Banach Lipschitz operator ideals such that $\I(X_0,N)\subset \J(X_0,N)$ for all pointed finite metric space $X_0$ and every finite dimensional Banach space $N$. Then $\I^{min}\leq \J^{min}$.
\end{proposition}
\begin{proof}
It follows since $\CLIP_{\J^{+}} \leq \CLIP_{\I^{+}}$, then $\bigwedge \CLIP_{\I^{+}}\leq \bigwedge \CLIP_{\J^{+}}$. An application of Proposition~\ref{prop: Minimal} completes the proof.
\end{proof}

Recall that from Theorem~\ref{Teo: Galois} we have that the maps $\psi\colon \LIN \rightarrow \LIP$ and $\phi\colon \LIP\rightarrow\LIN$ defined as
$$
\psi(\A)=\A\circ \Lip_0 \quad \mbox{and} \quad\phi(\I)=\I\cap \L,
$$
set up a Galois connection between $\LIN$ and $\LIP$. As a consequence of \cite[7.27]{DP02} we have that the composition $\phi\circ\psi:\LIN\to\LIN$ is a \textit{closure operator}. A map $c:P\to P$ is a \textit{closure operator} if for all $p$, $q\in P$ satisfies
\begin{enumerate}[\upshape (i)]
\item $p\leq c(p)$
\item if $p\leq q$, then $c(p)\leq c(q)$
\item $c(c(p))=c(p)$.
\end{enumerate}

Also, again by \cite[7.27]{DP02},  $\psi\circ\phi:\LIP^d\to\LIP^d$ is a closure operator where $\LIP^d$ is the dual of $\LIP$, that is the set $\LIP$ endowed with the order given by $\I\leq_d \J$ if and only if $\J\leq \I$. The next result characterize the minimal kernel of a Banach Lipschitz operator ideal of composition type.
\begin{theorem}\label{Teo:min}
Let $\A$ be a Banach operator ideal. Then
$$
(\A\circ \Lip_0)^{min}=\A^{min}\circ \Lip_0 \quad isometrically.
$$
\end{theorem}
\begin{proof}

First we show that there exists a Banach linear operator ideal $\B$ such that the equality $(\A\circ \Lip_0)^{min}=\B^{min}\circ \Lip_0$ holds. For this, consider the closure operator
$$
\psi \circ \phi \colon \CLIP_{(\A\circ \Lip_0)^{+}}^{d}\rightarrow \CLIP_{(\A\circ \Lip_0)^{+}}^{d}.
$$
To see that is well defined, take $\I \in \CLIP_{(\A\circ \Lip_0)^{+}}$. Then for  $X_0 \in \MFIN$ and $N \in \FIN$ we have $\I(X_0;N)\supset\A\circ \Lip_0(X_0;N)$, which, by the monotony of $\psi \circ \phi$, implies that $\psi \circ \phi (\I) (X_0;N)\supset \psi \circ \phi(\A\circ \Lip_0)(X_0;N)$. Since $\A\circ \Lip_0=\psi(\A)$, by Corollary~\ref{Coro: ban} we get that $\psi \circ \phi (\I) (X_0;N)\supset \A\circ \Lip_0(X_0;N)$. Then $\psi \circ \phi (\I) \in \CLIP_{(\A\circ \Lip_0)^{+}}$.

As the top element of $\CLIP_{(\A\circ \Lip_0){+}}^{d}$ coincides with the bottom element of $\CLIP_{(\A\circ \Lip_0)^{+}}$ which, by Proposition~\ref{prop: Minimal}, is $(\A\circ \Lip_0)^{min}$, by \cite[7.2]{DP02} $(\A\circ \Lip_0)^{min}$ is a closed element. That is, there exists $\I \in \CLIP_{(\A\circ \Lip_0)^{+}}$ such that $(\A\circ \Lip_0)^{min}=\psi\circ \phi(\I)$, meaning that
$$
(\A\circ \Lip_0)^{min}=(\I\cap \L)\circ \Lip_0 \quad isometrically.
$$
Now, we see that the last equality still valid if we consider the minimal kernel of the Banach linear operator ideal $(\I\cap \L)$ instead of itself. First, note that by the monotony, $(\I\cap \L)^{min}\circ \Lip_0\leq(\I\cap \L)\circ \Lip_0$. On the other hand, since $\Ae(X_0)$ is a finite dimensional space whenever $X_0$ is a finite metric space, then $(\I\cap \L)(\Ae(X_0);N)=(\I\cap \L)^{min}(\Ae(X_0);N)$ for any finite dimensional space $N$. Thus
$$
(\I\cap \L)\circ \Lip_0(X_0;N)=(\I\cap \L)^{min}\circ \Lip_0(X_0;N),
$$
and, as a consequence,  $(\I\cap \L)^{min}\circ \Lip_0\in\CLIP_{(\A\circ\Lip_0)^{+}}$. Since $(\I\cap \L)\circ \Lip_0$ is the bottom element of $\CLIP_{(\A\circ \Lip_0)^{+}}$, then $(\I\cap \L)^{min}\circ \Lip_0=(\I\cap \L)\circ \Lip_0$ isometrically.

It remains to show that $(\I\cap \L)^{min}=\A^{min}$. For this, as a consequence of \cite[Propostion~22.1]{DF93}, is enough to show that $(\I\cap \L)^{max}(M;N)=\A^{max}(M;N)$ for all finite dimensional spaces $M$ and $N$. Since the Banach Lipschitz ideals $(\I\cap \L)^{min}\circ \Lip_0$ and $\A\circ\Lip_0$ belongs to the complete lattice $\CLIP_{(\A\circ \Lip_0)}$, by Proposition~\ref{prop: Comp Maximal} we get the equality $(\I\cap \L)^{max}\circ \Lip_0=\A^{max}\circ\Lip_0$. Thus,
\begin{equation}\label{eq1}
((\I\cap \L)^{max}\circ \Lip_0)\cap \L=(\A^{max}\circ\Lip_0)\cap\L.
\end{equation}
Finally, if $M$ is a finite dimensional space, by  \cite[Theorem~3.1]{GK03} $M$ has the isometric Lipschitz-lifting property. Then combining~\eqref{eq1} and Proposition~\ref{Prop:lifting property} we get that $(\I\cap \L)^{max}(M;N)=\A^{max}(M;N)$ and the proof follows.
\end{proof}

As a direct consequence, we have
\begin{corollary}\label{coro: minimal circ lipo} Let $\A$ a minimal Banach linear operator ideal. Then $\A\circ \Lip_0$ is a minimal Banach Lipschitz operator ideal.
\end{corollary}

We finish with a brief discussion about the minimal kernel of the Lipschitz ideal
of $p$-nuclear and $p$-integral operators which, by Proposition~\ref{Prop: No compo}, are not of composition type. Recall that, for $1\leq p<\infty$, the ideals of $p$-nuclear and $p$-integral linear operators are related as $\N_p=(\I_p)^{min}$. However, the Banach Lipschitz operator ideal $\N_p^{L}$  (and hence $\I_p^{L}$) is not minimal. Indeed, if it were minimal, by Proposition~\ref{Imin mas chico} every Lipschitz $p$-nuclear operator must be approximable. But, by Proposition~\ref{Prop: Delta}, the Dirac map  $\delta\colon \mathds R\rightarrow \Ae(\mathds R)$  is not approximable and is $p$-nuclear (see above Proposition~\ref{Prop: Delta}).

\begin{proposition}\label{Prop: Min max Np}
Let $1\leq p <\infty$, then
$$
(\I_p^{L})^{min}=(\N_p^{L})^{min} \quad \mbox{and} \quad (\I_p^{L})^{max}=(\N_p^{L})^{max} \quad isometrically.
$$
\end{proposition}

\begin{proof}
By \cite[Theorem~4.1]{CheZhe12} we have that $\CLIP_{\I_p^L}=\CLIP_{\N_p^{L}}$. Then, the result follows from the definitions of minimal kernel and maximal hull of a Banach Lipschitz operator ideal.
\end{proof}
As a consequence, we have
\begin{proposition}\label{Prop: Inclusion N_p} Let $1\leq p <\infty$, then
$\N_p\circ \Lip_0 \leq (\N_p^{L})^{min}$ and the inclusion is strict. Moreover, $\N_p\circ \Lip_0 (\mathds R;\Ae(\mathds R))$ and $(\N_p^{L})^{min}(\mathds R,\Ae(\mathds R))$ differ.
\end{proposition}
\begin{proof}
Since $\N_p\circ \Lip_0\leq \N_p^{L}$, combining Proposition~\ref{prop: inclusion minimal} and Corollary~\ref{coro: minimal circ lipo} we have $\N_p\circ \Lip_0 \leq (\N_p^{L})^{min}$. Now, suppose that $\N_p\circ \Lip_0(\mathds{R};\Ae(\mathds{R})) = (\N_p^{L})^{min}(\mathds{R};\Ae(\mathds{R}))$. Since by Propositions~\ref{prop:maxmin} and~\ref{prop: Comp Maximal} we have that $((\N_p^{L})^{min})^{max}=(\N_p^{L})^{max}$ and $(\N_p\circ \Lip_0)^{max}=\I_p\circ \Lip_0$ we get that $(\N_p^{L})^{max}(\mathds{R};\Ae(\mathds{R}))=\I_p\circ \Lip_0(\mathds{R};\Ae(\mathds{R}))$.
This is a contradiction since, as the Dirac map $\delta_{\mathds{R}}\colon \mathds R \rightarrow \Ae(\mathds R)$ is Lipschitz $p$-nuclear, then $\delta_{\mathds{R}}\in(\N_p^L)^{max}(\mathds{R};\Ae(\mathds{R}))$ but  $\delta_{\mathds{R}}\notin\I_p\circ \Lip_0(\mathds{R};\Ae(\mathds{R}))$.
\end{proof}

Although $\N_p\circ \Lip_0$ and $(\N_p^{L})^{min}$ differ, the linear operators which belong to both ideals are, in most of the cases, the same.

\begin{proposition}\label{Prop: Np^min lin} Let $1\leq p <\infty$ and $E,F$ be Banach spaces such that $E$ is
separable and $F$ is a dual space with the metric approximation property. Then $(\N_p^{L})^{min}\cap \L (E;F)=\N_p(E;F)$ isometrically.
\end{proposition}
\begin{proof}
Take $E$ and $F$ Banach spaces as in the statement. By Proposition~\ref{Prop: Inclusion N_p}, we have the inclusions $\N_p\circ \Lip_0 (E;F)\subset (\N_p^{L})^{min}(E;F)\subset \N_p^{L} (E;F)$. Since by Proposition~\ref{Prop:lifting property} $(\N_p\circ \Lip_0)\cap\L (E;F)=\N_p(E;F)$ and, by \cite[Theorem~2.1]{CheZhe12}, $\N_p^{L}\cap\L (E;F)=\N_p(E;F)$, the result follows.
\end{proof}
As a consequence, we obtain that $(\N_p^L)^{min}$ is a minimal Banach Lipschitz operator ideal which is not of composition type.

\begin{proposition}\label{Prop: Np no compo} Let $1\leq p<\infty$. Then $(\N_p^{L})^{min}$ is not a Banach Lipschitz operator ideal of composition type.
\end{proposition}
\begin{proof}  If $(\N_p^{L})^{min}$ were of composition type, then by Proposition~\ref{Prop: Criterion} we have that $(\N_p^{L})^{min}=((\N_p^{L})^{min}\cap \L)\circ \Lip_0$. In particular $(\N_p^{L})^{min}(\mathds R,\Ae(\mathds R))= ((\N_p^{L})^{min}\cap \L)\circ \Lip_0(\mathds R,\Ae(\mathds R))$. Now, as $\Ae(\mathds R)=L_1$, by the above proposition we get that $\N_p\circ \Lip_0 (\mathds R;\Ae(\mathds R))=(\N_p^{L})^{min}(\mathds R,\Ae(\mathds R))$, which contradict Proposition~\ref{Prop: Np^min lin}. The proof finish.
\end{proof}

As a consequence of Proposition~\ref{Prop: Np no compo}, every Banach Lipschitz ideal belonging to the sublattice $\CLIP_{\N_p^L}$ must be not of composition type. Indeed, if $\A\circ\Lip_0\in\CLIP_{\N_p^L}$, then $\CLIP_{\N_p^L}=\CLIP_{\A\circ\Lip_0}$ which implies that $\A^{min}\circ\Lip_0=(\N_P^L)^{min}$. In general, we have.

\begin{proposition}
  Let $\I$ be a Banach Lipschitz operator ideal such that $\I^{min}$ or $\I^{max}$ is not of composition type. Then, every Banach Lipschitx operator ideal $\J\in\CLIP_{\I}$ is not of composition type. In particular, $\I$ is not of composition type.
\end{proposition}

We do not know if the converse of the above proposition holds.

\begin{problem} Let $\I$ be a Banach Lipschitz operator ideal. Suppose that $\I^{min}$ or $\I^{max}$ is of composition type. Does $\I$ must be of composition type?
\end{problem}

%By \cite{???} we now that for $\A \in \LIN$,
%$$
%\bigvee_{\J \in \CLIP_{\Psi(\A)}} \J=\Psi({\A^{max}})=\Psi (\bigvee_{\B \in \CLIN_{\A}} \B)
%$$.
%Now we have:
%\begin{proposition} Let $\A \in \LIN$. Then
%$$
%\bigwedge_{\J \in \CLIP_{\Psi(\A)}}=\Psi({\A^{min})=\Psi (\bigwedge_{\B \in \CLIN_{\A}} \B)
%$$
%\end{proposition}
%\begin{proof}
%First, note that $\bigwedge_{\J \in \CLIP_{\Psi(\A)}}=\bigvee_{\J \in \CLIP^{d}_{\Psi(\A)}}$.
%Now, by \cite[7,27]{Davey-Priestley}, the operators $\psi \circ \phi \colon \LIP^{d}\rightarrow \LIP^{d}$ is a closure operator. In particular,
%\end{proof}
\end{document}